\documentclass[11pt]{article}
\usepackage{mathrsfs}
\usepackage{bbm}

\usepackage{verbatim}

\usepackage{amsthm}
\usepackage{amsmath}
\usepackage{amssymb}

\usepackage{graphicx}
\usepackage{epstopdf}
\usepackage{titling}

\usepackage{enumerate}

\usepackage{algpseudocode}
\usepackage{algorithm}
\usepackage{algorithmicx}
\usepackage{tikz}

\newtheorem{theorem}{Theorem}[section]

\newtheorem{lemma}[theorem]{Lemma}

\newtheorem{corollary}[theorem]{Corollary}

\newtheorem{definition}[theorem]{Definition}
\newtheorem{question}[theorem]{Question}

\newcommand{\Ind}{\textup{Ind}}
\newcommand{\Cop}{\textup{Cop}}

\newcommand{\I}{[0,1]}
\newcommand{\pa}[1]{d(\overline{K_k};#1)}
\newcommand{\pc}[1]{d(K_k;#1)}

\begin{document}

\title{On the densities of cliques and independent sets in graphs}

\author{Hao Huang\thanks{
School of Mathematics, Institute for Advanced Study, Princeton 08540. Email: {\tt huanghao@math.ias.edu}. Research supported in part by NSF grant DMS-1128155.
}
\and Nati Linial\thanks{School of Computer Science and engineering, The Hebrew University of Jerusalem, Jerusalem 91904, Israel. Email: {\tt nati@cs.huji.ac.il}. Research supported in
part by the Israel Science Foundation and by a USA-Israel BSF grant.}
\and Humberto Naves\thanks{Department of Mathematics, UCLA, Los Angeles, CA 90095. Email: {\tt hnaves@math.ucla.edu}.}
\and Yuval Peled\thanks{School of Computer Science and engineering, The Hebrew University of Jerusalem, Jerusalem 91904, Israel. Email: {\tt yuvalp@cs.huji.ac.il}}
\and Benny Sudakov\thanks{Department of Mathematics, ETH, 8092 Zurich, Switzerland and Department of Mathematics, UCLA, Los Angeles, CA 90095.
Email: bsudakov@math.ucla.edu. Research supported in part by SNSF grant 200021-149111 and by a USA-Israel BSF grant.}}

\date{}

\maketitle

\setcounter{page}{1}

\vspace{-2em}
\begin{abstract}
Let $r, s \ge 2$ be integers. Suppose that the number of blue $r$-cliques in a red/blue coloring of the edges of the complete graph $K_n$ is known and fixed. What is
the largest possible number of red $s$-cliques under this assumption? The well known Kruskal-Katona theorem answers this question for $r=2$ or $s=2$.
Using the shifting technique from extremal set theory together with some analytical arguments, we resolve this problem in general and prove that in the extremal
coloring either the blue edges or the red edges form a clique.
\end{abstract}

\section{Introduction} \label{section_introduction}

As usual we denote by $K_s$ the complete graph on $s$ vertices and by $\overline{K}_s$ its complement, the edgeless graph on $s$ vertices. By the celebrated Ramsey's
theorem, for every two integers $r, s$ every sufficiently large graph  must contain $K_r$ or $\overline{K}_s$. Tur\'an's theorem can be viewed as a
quantitative version of the case $s=2$. Namely, it shows that among all $\overline{K}_r$-free $n$-vertex graphs, the graph with the least number of
$K_2$ (edges) is a disjoint union of $r-1$ cliques of nearly equal size.
More generally, one can ask the following question. Fix two graphs $H_1$ and $H_2$, and suppose that we know the number of induced copies of
$H_1$ in an $n$-vertex graph $G$. What is the maximum (or minimum) number of induced copies of $H_2$ in $G$? In its full generality, this problem seems currently out
of reach, but some special cases already have important implications in combinatorics, as well as other branches of mathematics and computer science.

To state these classical results, we introduce some notation. Adjacency between vertices $u$ and $v$ is denoted by $u \sim v$, and the neighbor set of $v$ is denoted
by $N(v)$. If necessary, we add a subscript $G$ to indicate the relevant graph. The collection of induced copies of a $k$-vertex graph $H$ in an $n$-vertex graph $G$
is denoted by $\Ind(H; G)$, i.e.
$$\Ind(H; G) := \{X \subseteq V(G): G[X] \simeq H\}$$ and the \textit{induced $H$-density} is defined as$$d(H; G) := \frac{|\Ind(H;
G)|}{\binom{n}{k}}.$$
In this language, Tur\'an's theorem says that if $d(K_r;G)=0$ then  $d(K_2; G)\le 1-\frac{1}{r-1}$ and this bound is tight.
For a general graph $H$, Erd\H{o}s and Stone
\cite{erdos-stone} determined $\max d(K_2; G)$ when $d(H;G)=0$ and showed that the answer depends only on the chromatic number of $H$. Zykov \cite{zykov}
extended Tur\'an's theorem in a different direction. Given integers $2 \le r<s$, he proved that if $d(K_s;G)=0$ then $d(K_r; G) \le \frac{(s-1) \cdots
(s-r)}{(s-1)^r}$. The balanced complete $(s-1)$-partite graphs show that this bound is also tight.

For fixed integers $r<s$, the Kruskal-Katona theorem \cite{katona, kruskal} states that if $d(K_r; G)=\alpha$ then  $d(K_s; G) \le \alpha^{s/r}$. Again, the
bound is tight and is attained when $G$ is a clique on some subset of the vertices. On the other hand, the problem of {\em minimizing} $d(K_s; G)$ under the same
assumption is much more difficult. Even the case $r=2$ and $s=3$ has remained unsolved for many years until it was recently answered by Razborov \cite{razborov} using
his newly-developed flag algebra method. Subsequently, Nikiforov \cite{nikiforov} and Reiher \cite{reiher} applied complicated analytical techniques to solve the
cases $(r,s)=(2,4)$, and ($r=2$, arbitrary $s$), respectively.

In this paper, we study the following natural analogue of the Kruskal-Katona theorem. Given $d(\overline{K}_r; G)$, how large can $d(K_s; G)$ be? For integers $a \ge b > 0$ we let
$Q_{a,b}$ be the $a$-vertex graph whose edge set is a clique on some $b$ vertices. The complement of this graph is denoted by $\overline Q_{a,b}$. Let $\mathcal{Q}_a$ denote the family
of all graphs $Q_{a,b}$ and its complement $\overline Q_{a,b}$ for $0 < b \le a$. Note that for $r=2$ or $s=2$, the Kruskal-Katona theorem implies that the extremal graph comes from $\mathcal{Q}_n$. Our first theorem shows that a similar statement holds for all $r$ and $s$.

\begin{theorem}
\label{maintheorem}
Let $r, s \ge 2$ be integers and suppose that $d(\overline{K}_r; G) \geq p$ where $G$ is an $n$-vertex graph and $0 \le p \le 1$. Let $q$ be the unique root
of $q^r+rq^{r-1}(1-q)=p$ in $[0,1]$. Then $d(K_s;G) \le M_{r,s,p} + o(1)$, where
\[
M_{r,s,p} := \max \{(1-p^{1/r})^s + sp^{1/r}(1-p^{1/r})^{s-1}, (1-q)^s\}.
\]
Namely, given $d(\overline{K}_r; G)$, the maximum of $d(K_s; G)$ (up to $\pm o_n(1)$) is attained in one of two graphs, (or both), one of the form $Q_{n,t}$ and another $\overline Q_{n,t'}$.
\end{theorem}

We obtain as well a {\em stability version} of Theorem \ref{maintheorem}. Two $n$-vertex graphs $H$ and $G$ are \textit{$\epsilon$-close} if it is possible to obtain $H$ from $G$ by adding or deleting at most $\epsilon n^2$ edges. As the next theorem shows, every near-extremal
graph $G$ for Theorem \ref{maintheorem} is $\epsilon$-close to a specific member of $\mathcal{Q}_n$.

\begin{theorem}
\label{stabilitytheorem}
Let $r, s \ge 2$ be integers and let $p \in [0,1]$. For every $\epsilon > 0$, there exists $\delta > 0$ and an integer $N$ such that every $n$-vertex graph $G$ with $n > N$ satisfying $d(\overline{K}_r;G) \ge p$ and
$|d(K_s;G) - M_{r,s,p}| \le \delta$, is $\epsilon$-close to some graph in $\mathcal{Q}_n$.
\end{theorem}

\begin{figure}[h!]
  \centering
    \includegraphics[width=0.5\textwidth]{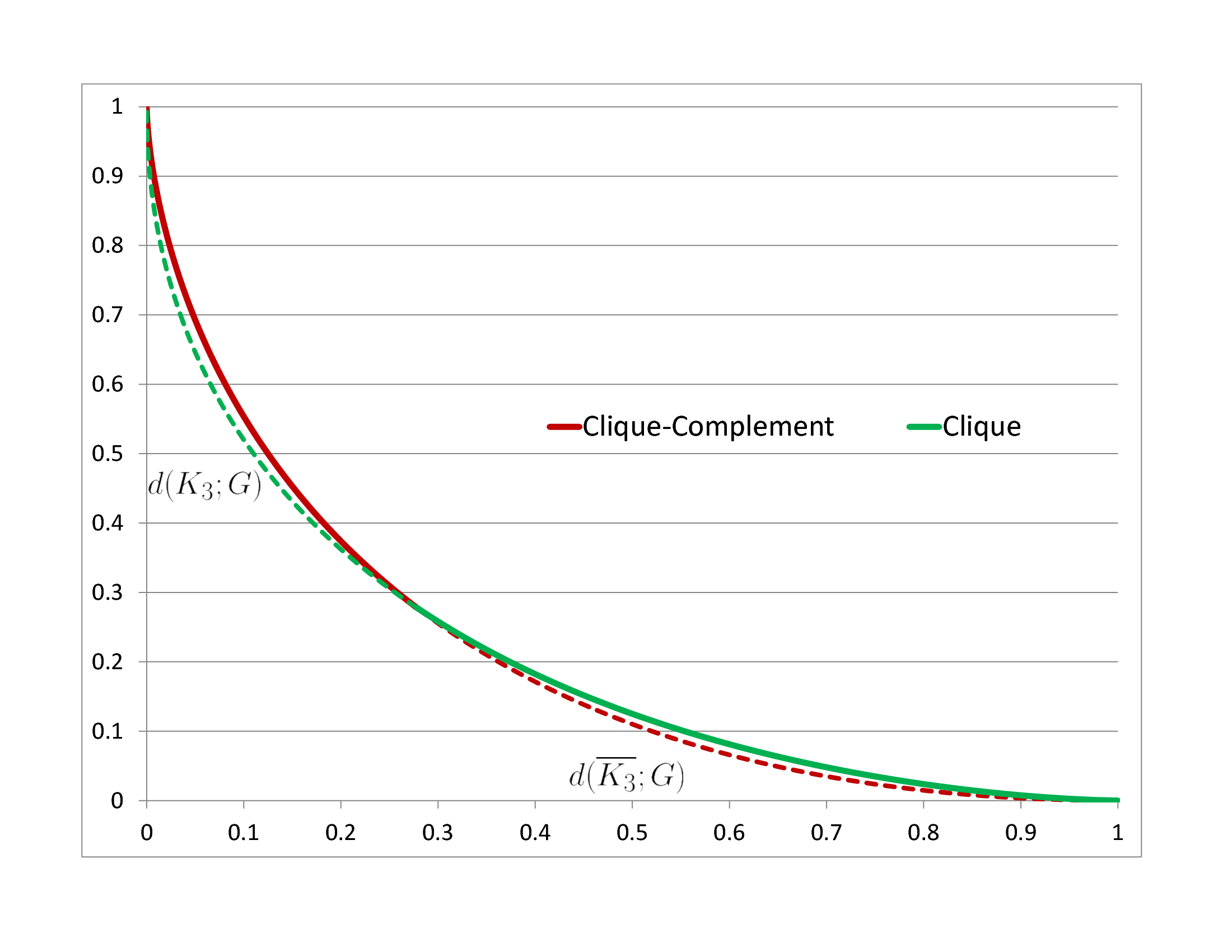}
    \caption{Illustration for the case $r=s=3$. The green curve is $(d(\overline K_3;Q_{n,\theta n}), d(K_3;Q_{n,\theta n}))$ for $\theta\in\I$, and the red  curve defined the same with $\overline Q_{n,\theta n}$. The maximum between the curves is the extremal function in Theorem \ref{maintheorem}. The intersection of the curves represents the solution of the max-min problem in Theorem~\ref{thm:max_min}}
\end{figure}

Rather than talking about an $n$-vertex graph and its complement, we can consider a two-edge-coloring of $K_n$. A quantitative version of Ramsey Theorem
asks for the minimum number of monochromatic $s$-cliques over all such colorings. Goodman \cite{goodman} showed that for $r=s=3$, the optimal
answer is essentially given by a random two-coloring of $E(K_n)$. In other words, $\min_G d(K_3; G) + d(\overline{K}_3; G) = 1/4-o(1)$. Erd\H{o}s \cite{erdos-false} conjectured
that the same random coloring also minimizes $d(K_r; G) + d(\overline{K}_r; G)$ for all $r$, but this was refuted by Thomason \cite{thomason} for all $r \ge 4$. A
simple consequence of Goodman's inequality is that $\min_G \max \{d(K_3; G), d(\overline{K}_3; G)\}=1/8$. The following construction by Franek and R\"odl \cite{franek-rodl}
shows that the analogous statement for $r \ge 4$ is again false. Let $H$ be a graph with vertex set $[2]^{13}$, the collection of all $8192$ binary vectors of length $13$. Two vertices are adjacent if the Hamming distance between the corresponding binary vectors
is a number in $\{1, 4, 5, 8, 9, 11\}$. Let $G$ be obtained from $H$ by replacing each vertex with a clique of size $n$, and
every edge with a complete bipartite graph. The number of $K_4$ and $\overline{K}_4$ in $G$ can be easily expressed in terms of the parameters of  $H$ (see \cite{franek-rodl}),
for large enough $n$ one can show that $d(K_4; G)<0.99\cdot \frac{1}{64}$ and $d(\overline{K}_4; G) <0.993\cdot \frac{1}{64}$.

While the min-max question remains at present very poorly understood, we succeeded to completely answer the max-min version of this problem.

\begin{theorem}
\label{thm:max_min}
$$\max_G \min \{d(K_r; G), d(\overline{K}_r; G)\} = \rho^r +o(1),$$
where $\rho$ is the unique root in $[0,1]$ of the equation $\rho^r = (1-\rho)^r + r\rho(1-\rho)^{r-1}$.
\end{theorem}
This theorem follows easily from Theorem \ref{maintheorem}. Moreover, using Theorem \ref{stabilitytheorem}, we can also show that for every $\epsilon>0$ there is a $\delta>0$ such
that
every $n$-vertex graph $G$ with $\min \{d(K_r; G), d(\overline{K}_r; G)\} > \rho^r-\delta$ is $\epsilon$-close to a clique of size $\rho n$ or to the complement of this graph.

Here we prove these theorems using the method of shifting. In the next section we describe this well-known and useful technique in extremal set theory. Using shifting, we show how to reduce the problem to \textit{threshold graphs}. Section \ref{section_main} contains the proof of our main result for threshold graphs and section \ref{section_stability} contains the proof of the stability result. In Section \ref{section_shift} we sketch a second proof for the case $r=s$, based on a different representation of threshold graphs. We make
a number of comments on the analogous problems for hypergraphs in Section~\ref{section_hyper}. We finish this paper with some concluding remarks
and open problems.

\section{Shifting} \label{section_shifting}
{\em Shifting} is one of the most important and widely-used tools in extremal set theory. This method allows one to reduce many extremal problems to more structured instances
which are usually easier to analyze. Our
treatment is rather shallow and we refer the reader to Frankl's survey article \cite{frankl-survey} for a fuller account.

Let $\mathcal{F}$ be a family of subsets of a finite set $V$, and let $u, v$ be two distinct elements of $V$. We define the \textit{$(u, v)$-shift map} $S_{u\to v}$ as follows: for
every $F \in \mathcal{F}$, let
\[ S_{u\to v}(F, \mathcal{F}) := \left\{\begin{array}{ll} (F \cup \{v\})\setminus \{u\} & \text{if } u \in F, v\not \in F \text{ and } (F \cup
\{v\})\setminus \{u\} \not\in \mathcal{F}, \\ F & \text {otherwise.} \end{array}\right. \] We define the $(u,v)$-shift of $\mathcal{F}$,
to be the following family of subsets of $V$: $S_{u\to v}(\mathcal{F}) := \{ S_{u\to v}(F, \mathcal{F}) : F \in \mathcal{F}\}$. We observe that $|S_{u\to
v}(\mathcal{F})|=|\mathcal{F}|$. In this context, one may think of
$\mathcal{F}$ as a hypergraph over $V$. When all sets in $\mathcal F$ have cardinality $2$ this is a graph with vertex set $V$. As the next lemma shows, shifting of graph does not
reduce the number of $l$-cliques in it for every $l$. Recall that $\Ind(K_l;G)$ denotes the collection of all cliques of size $l$ in $G$.

\begin{lemma} \label{lemma_increase} For every integer $l>0$, every graph $G$, and every $u\neq v \in V(G)$ there holds
\[ S_{u\to v}(\Ind(K_l;G))\subseteq \Ind(K_l;S_{u\to v}(G)). \]
\end{lemma}
\begin{proof} Let $A=S_{u\to v}(B,G)$, where $B$ is an $l$-clique in $G$. First, consider the cases when
$u\notin B$ or both $u, v\in B$ or $B\setminus\{u\}\cup\{v\}$ is also a clique in $G$. Then $A=B$ and we need to show that $B$ remains a clique after shifting. Which edge in $B$ can be
lost by shifting? It must be some edge $uw$ in $B$ that gets replaced by the non-edge $vw$ (otherwise we can not shift $uw$). Note that $vw$ is not in $B$, since $B$ is a clique.
Hence $u, w\in B$ and $v\not\in B$. But
then $B\setminus\{u\}\cup\{v\}$ is not a clique, contrary to our assumption.

In the remaining case when $u\in B$, $v\notin B$ and $B\setminus\{u\}\cup\{v\}$ is not a clique in $G$, we need to show that $A=B\setminus\{u\}\cup\{v\}$ is a clique after
shifting $S_{u\to v}(G)$. Every pair of vertices in $A\setminus\{v\}$ belongs to $B$ and the edge they span is not affected by the shifting.
So consider $v\not = w\in A$. If $vw \in E(G)$, this edge remains after shifting. If, however, $vw \notin E(G)$, note that $uw \in E(G)$ since both vertices belong to the clique
$B$. In this case $vw=S_{u\to v}(uw, G)$ and the claim is proved. \end{proof}

Since shifting edges from $u$ to $v$ is equivalent to shifting non-edges from $v$ to $u$, it is immediate that $S_{u\to v}(\Ind(\overline{K_l};G))\subseteq
\Ind(\overline{K_l};S_{u\to v}(G))$. Therefore we obtain the following corollary.

\begin{corollary}
\label{corollary_increase}
Let $G$ be a graph, let $H = S_{u\to v}(G)$ and let $l$ be a positive integer. Then
$$d(K_l; H) \ge d(K_l; G) \textrm{~~and~~~} d(\overline{K}_l; H) \ge d(\overline{K}_l; G).$$
\end{corollary}

We say that vertex $u$ \textit{dominates} vertex $v$ if $S_{v\to u}(\mathcal{F})=\mathcal{F}$. In the case when $ \cal F$ is a set of edges of $G$, this implies
that every $w\not = u$ which is adjacent to $v$ is also adjacent to $u$. If $V=[n]$, we say that a family $\mathcal{F}$ is \textit{shifted} if $i$ dominates $j$ for every $i<j$. Every
family can be made shifted by repeated applications of shifting operations $S_{j\to i}$ with $i<j$. To see this note that a shifting operation that changes $\mathcal F$ reduces the
following non-negative potential function $\sum_{A\in\mathcal{F}}\sum_{i\in A}i$. As Corollary \ref{corollary_increase} shows, it suffices to prove Theorem \ref{maintheorem} for
shifted graphs.

In Section \ref{section_main} we use the notion of \textit{threshold graphs}.
There are several equivalent ways to define threshold graph (see \cite{chvatal-hammer}), and we adopt the
following definition.

\begin{definition}\label{definition_threshold}
We say that $G=(V,E)$ is a threshold graph if there is an ordering of $V$ so that every vertex is adjacent to either all or none of the preceding vertices.
\end{definition}

\begin{lemma}
\label{lemma_threshold}
A graph is shifted if and only if it is a threshold graph.
\end{lemma}
\begin{proof}
Let $G$ be a shifted graph. We may assume that $V=[n]$, and $i$ dominates $j$ in $G$ for every $i<j$. Consider the following order of vertices,
$$
...,\;3,\;N_G(2)\backslash N_G(3),\;2,\;N_G(1)\backslash N_G(2),\;1,\; V\backslash N_G(1)\;,
$$
where the vertices inside the sets that appear here are ordered arbitrarily.
We claim that this order satisfies Definition~\ref{definition_threshold}. First, every vertex $v\notin  N_G(1)$ is isolated.
Indeed, if $u\sim v$, then necessarily $v\sim 1$, since $1$ dominates $u$. Therefore, vertex $1$ and its non-neighbors satisfy the condition in the definition.
The proof that $G$ is threshold proceeds by induction applied to $G[N_G(1)]$.

Conversely, let $G$ be a threshold graph. Let $v_1, v_2, \ldots, v_n$ be an ordering of $V$ as in Definition~\ref{definition_threshold}. We say that a vertex is good (resp. bad) if it is adjacent to all (none) of its preceding vertices. Consider two vertices $v_i$ and $v_j$. It is straightforward to show that $v_i$ dominates $v_j$ if either (1) $v_i$ is good and $v_j$ is bad, (2) they are both good and $i>j$ or (3) they are both bad and $i<j$. Therefore we can reorder the vertices by first placing the good vertices in reverse order followed by the bad vertices in the regular order. This new ordering demonstrates that $G$ is shifted.
\end{proof}

\section{Main result} \label{section_main}
In this section, we prove Theorem \ref{maintheorem}. It will be convenient to reformulate the theorem, in a way that is analogous to the Kruskal-Katona theorem.
\begin{theorem}
\label{maintheorem_unrestricted}
Let $r, s \ge 3$ be integers and let $a,b > 0$ be real numbers. The maximum (up to $\pm o_n(1)$) of the function $f(G):= \min\{a\cdot d(K_s; G), b \cdot d(\overline{K}_r; G)\}$ over all $n$-vertex graphs
is attained in one of two graphs, (or both), one of the form $Q_{n,t}$ and another $\overline Q_{n,t'}$.
In particular, $f(G) \le \max \{a \cdot \alpha^{s}, b \cdot \beta^{r}\} + o(1)$, where $\alpha$ is the unique
root in $[0,1]$ of $a \cdot \alpha^s = b \cdot [(1-\alpha)^r + r \alpha (1-\alpha)^{r-1}]$ and $\beta$ is the unique root in $[0,1]$ of $b \cdot \beta^r = a \cdot [(1-\beta)^s + s \beta
(1-\beta)^{s-1}]$.
\end{theorem}

We turn to show how to deduce Theorem \ref{maintheorem} from Theorem \ref{maintheorem_unrestricted}. We assume that $r,s \ge 3$, since the other cases follow from
Kruskal-Katona theorem.

\begin{proof}[Proof of Theorem \ref{maintheorem}]
Let $M$ be the maximum of $d(K_s;G)$ over all graphs $G$ on $n$ vertices with $d(\overline{K}_r;G) \geq p$. Fix such an
extremal $G$ with $d(\overline{K}_r;G) =  p'\geq p$
and $d(K_s;G) = M$. Now apply Theorem \ref{maintheorem_unrestricted} with $a = p$ and $b = M$ and the same $n$, $r$ and $s$. The extremal graph $G'$
that Theorem \ref{maintheorem_unrestricted} yields, satisfies
\[
f(G') \ge f(G) = \min\{a\cdot d(K_s;G), b\cdot d(\overline{K}_r;G)\} = p \cdot M,
\]
hence $d(K_s;G') \ge M$ and $d(\overline{K}_r;G')\ge p$. Therefore, the same $G'$ is extremal for Theorem \ref{maintheorem} as well and we know that the maximum in this theorem is achieved asymptotically by a graph of $\mathcal Q_n$.

Note that we can always assume that in the extremal graph
$d(\overline{K}_r;G')=p$ since otherwise we can add edges to $G'$ without decreasing $d(K_s;G')$ until $d(\overline{K}_r;G')=p$ is obtained. Therefore the maximum is attained either by a graph of the form $\overline Q_{n,p^{1/r}n}$ or by $Q_{n,(1-q)n}$, where $q^r+rq^{r-1}(1-q)=p$. This implies that asymptotically the maximum in Theorem \ref{maintheorem} is indeed
$$M_{r,s,p} = \max \{(1-p^{1/r})^s + sp^{1/r}(1-p^{1/r})^{s-1}, (1-q)^s\}.$$
\end{proof}

By Corollary \ref{corollary_increase} and Lemma \ref{lemma_threshold}, $f(G)$ is maximized by a threshold graph. We turn to prove Theorem \ref{maintheorem_unrestricted} for threshold graphs. Let $G$ be a threshold graph on an ordered vertex set $V$, as in Definition~\ref{definition_threshold}. There exists an integer $k>0$, and a partition $A_1,\ldots,A_{2k}$ of $V$ such that
\begin{enumerate}
\item If $v\in A_i$ and $u\in A_j$ for $i<j$, then $v<u$.
\item Every vertex in $A_{2i-1}$ (respectively $A_{2i}$) is adjacent to all (none) of its preceding vertices.
\end{enumerate}
Let $x_i=\frac{|A_{2i-1}|}{|V|}$ and $y_i=\frac{|A_{2i}|}{|V|}$. Clearly $\sum_{i=1}^k (x_i + y_i) = 1$. Up to a negligible error-term,
\begin{align*}
d(K_s;G)=p(\mathbf{x},\mathbf{y}) &:= \left(\sum_{i=1}^{k} x_i\right)^s + s \cdot \sum_{i = 1}^{k-1} \left[y_i\cdot \left(\sum_{j=i + 1}^k x_j\right)^{s-1}\right], \\
d(\overline K_r;G)=q(\mathbf{x},\mathbf{y}) &:= \left(\sum_{i=1}^{k} y_i\right)^r + r \cdot \sum_{i = 1}^{k} \left[x_i\cdot \left(\sum_{j=i}^k y_j\right)^{r-1}\right].
\end{align*}
Where $\mathbf{x} = (x_1,x_2,\ldots, x_k)$ and $\mathbf{y} = (y_1,y_2,\ldots, y_k)$. Occasionally, $p$ will be denoted by $p_s$ and $q$ by $q_r$ to specify the parameter of these functions.

Our problem can therefore be reformulated as follows. For
given integers $k\ge 2$, $r,s\ge 3$ and real $a,b>0$, let $W_k \subseteq \mathbb{R}^{2k}$ be the set
\[
W_k := \left\{(x_1,x_2,\ldots, x_k,y_1,y_2,\ldots,y_k)\in \mathbb{R}^{2k} :
x_i,y_i \ge 0 \text{ for all $i$ and }
\sum_{i=1}^k (x_i + y_i) = 1\right\}.
\]
Let $p, q : W_k \to \mathbb{R}$ be the two homogeneous polynomials defined above, We are interested in maximizing the real function $$\varphi(\mathbf{x},\mathbf{y}) := \min \{a \cdot p(\mathbf{x},\mathbf{y}), b \cdot q(\mathbf{x},\mathbf{y})\}.$$ This problem is well defined since $W_k$ is compact and $\varphi$ is continuous.

We say that $(\mathbf x,\mathbf y)\in W_k$ is \emph{non-degenerate}
if the set of zeros in the sequence $(y_1, x_2, y_2,\ldots,  x_k, y_k)$,
with $x_1$ omitted, forms a suffix. If $(\mathbf x,\mathbf y)\in W_k$ is degenerate, then there is a non-degenerate $(\mathbf x',\mathbf y')\in W_k$ with
$\varphi(\mathbf{x},\mathbf{y})=\varphi(\mathbf{x}',\mathbf{y}')$. Indeed, if $y_i=0$ and $x_{i+1}\ne 0$ for some $1\le i <k$, let $(\mathbf x',\mathbf y')\in W_{k-1}$ be defined by
\[
\mathbf x' = (x_1,\ldots,x_{i-1},x_i+x_{i+1},x_{i+2},\ldots,x_k)
\]
\[
\mathbf y' = (y_1,\ldots,y_{i-1},y_{i+1},\ldots,y_k)
\]

It is easy to verify that $p(\mathbf x,\mathbf y)=p(\mathbf x',\mathbf y')$ and $q(\mathbf x,\mathbf y)=q(\mathbf x',\mathbf y')$. By induction on $k$, we assume that $(\mathbf x',\mathbf y')$ is non-degenerate, and by padding $\mathbf x'$ and $\mathbf y'$ with a zero, the claim is proved. The case $x_i=0$ and $y_i\ne 0$ is proved similarly. In particular, $\varphi$ has a non-degenerate maximum in $W_k$.

Our purpose is to show that the original problem is optimized by graphs from $\mathcal{Q}_n$. This translates to the claim that a non-degenerate $(\mathbf{x},\mathbf{y})$ that
maximizes $\varphi$ is supported only on either $x_1,y_1$ or $y_1,x_2$, which corresponds to either a clique $Q_{n,t}$ or a complement of a clique $\overline{Q}_{n,t}$, respectively.

\begin{lemma}
\label{lemma_mainLemma}
Let $(\mathbf{x},\mathbf{y}) \in W_k$ be a non-degenerate maximum of $\varphi$. If $x_1>0$, then for every $i\ge 2$, $x_i=y_i=0$. On the other hand, if $x_1=0$ then $y_i=0$ for every
$i\ge 2$, and $x_i=0$ for every $i\ge 3$.
\end{lemma}
\begin{proof}
We note first that the second part of the lemma is implied by the first part. Define $\mathbf{x}'$ by
\[
x'_i := \left\{\begin{array}{ll}
x_{i+1} & \text{ if } i < k, \\
0  & \text{ if } i = k.
\end{array}\right.
\]
Clearly, if $x_1=0$, then $p_s(\mathbf{x,y})=q_s(\mathbf{y,x'})$, $q_r(\mathbf{x,y})=p_r(\mathbf{y,x'})$, and $$\varphi'(\mathbf{y},\mathbf{x'}) := \min \{b \cdot p_r(\mathbf{y},\mathbf{x'}), a \cdot q_s(\mathbf{y},\mathbf{x'})\}=\varphi(\mathbf{x},\mathbf{y}).$$ Since $\varphi$ attains its maximum when $x_1=0$, maximizing it is equivalent to maximizing $\varphi'(\mathbf{y,x'})$. Since $(\mathbf{x},\mathbf{y})$ is non-degenerate, $y_1>0$, and applying the first part of Lemma~\ref{lemma_mainLemma} for $\varphi '(\mathbf y,\mathbf x ')$ finishes the proof, by  obtaining that for every $i\geq 2$, $y_i=x'_i=0$.

The first part of Lemma~\ref{lemma_mainLemma} is proved in the following lemmas. We successively show that $x_3=0$, then $y_2=0$ and finally $x_2=0$.

\end{proof}

Here is a local condition that maximum points of $\varphi$ satisfy.
\begin{lemma}
\label{lemma_equality}
If $\varphi$ takes its maximum at a non-degenerate $(\mathbf{x}, \mathbf{y}) \in W_k$, then $a \cdot p(\mathbf{x}, \mathbf{y}) = b\cdot q(\mathbf{x}, \mathbf{y})$.
\end{lemma}
\begin{proof}
Note that $0<y_1<1$, since $(\mathbf{x},\mathbf{y})\in W$ is non-degenerate. We consider two perturbations of the input, one of which increases $p(\mathbf{x},\mathbf{y})$, and the other increases $q(\mathbf{x},\mathbf{y})$. Consequently, if $a\cdot p(\mathbf{x},\mathbf{y}) \ne b\cdot q(\mathbf{x},\mathbf{y})$, by applying the appropriate perturbation, we increase the smallest between $a\cdot p(\mathbf{x},\mathbf{y})$ and $b\cdot q(\mathbf{x},\mathbf{y})$, thus increasing $\min\{a\cdot p(\mathbf{x},\mathbf{y}), b\cdot q(\mathbf{x},\mathbf{y})\}$, contrary to the maximality assumption.

To define the perturbation that increases $p$, let $\mathbf{x'}=\mathbf{x}+t\mathbf{e_1}$ and $\mathbf{y'}=\mathbf{y}-t\mathbf{e_1}$,
where $0<t<y_1$, and $\mathbf{e_1}$ is the first unit vector in $\mathbb{R}^k$. Then, $(\mathbf{x'},\mathbf{y'}) \in W$ and
\[
\frac{\partial p(\mathbf{x'},\mathbf{y'})}{\partial t}=
s \left(t+\sum_{i=1}^{k} x_i\right)^{s-1} - s \cdot   \left(\sum_{j=2}^k x_j\right)^{s-1} > 0
\]
as claimed.

In order to increase $q$, consider two cases. If $x_1=0$, let $\mathbf{x'}=\mathbf{x}-t\mathbf{e_2}$ and $\mathbf{y'}=\mathbf{y}+t\mathbf{e_1}$,
where $0<t<x_2$. Then, $(\mathbf{x'},\mathbf{y'}) \in W$ and
\[
\frac{\partial q(\mathbf{x'},\mathbf{y'})}{\partial t}=
r \left(t+\sum_{i=1}^{k} y_i\right)^{r-1} - r \cdot   \left(\sum_{j=k}^n y_j\right)^{r-1} > 0.
\]
If $x_1>0$, we let $\mathbf{x'}=\mathbf{x}-t\mathbf{e_1}$ and $\mathbf{y'}=\mathbf{y}+t\mathbf{e_1}$,
where $0<t<x_1$. Then,
\[
\frac{\partial q(\mathbf{x'},\mathbf{y'})}{\partial t}=
r(x_1-t)(r-1)\left(t+\sum_{i=1}^{k} y_i\right)^{r-2}> 0.\]

\end{proof}

\begin{lemma}
\label{lemma_small0}
If $(\mathbf{x},\mathbf{y}) \in W_k$ is a non-degenerate maximum of $\varphi$ with $x_1>0$, then $x_3=0$.
\end{lemma}
\begin{proof}
Suppose, that $x_3 > 0$ and let $1 \le l \le m \le k$. Then
\begin{align*}
\frac{\partial p}{\partial x_l} &= s \cdot \left(\sum_{i=1}^{k} x_i\right)^{s-1} + s(s-1) \cdot \sum_{i = 1}^{l-1} \left[y_i\cdot \left(\sum_{j=i + 1}^k x_j\right)^{s-2}\right], \\
\frac{\partial q}{\partial x_l} &= r \cdot \left(\sum_{j=l}^k y_j\right)^{r-1},
\end{align*}
and
\begin{align*}
\frac{\partial^2 p}{\partial x_l \partial x_m} &= s(s-1)\cdot \left(\sum_{i=1}^{k} x_i\right)^{s-2} + s(s-1)(s-2) \cdot \sum_{i = 1}^{l-1} \left[y_i\cdot \left(\sum_{j=i + 1}^k x_j\right)^{s-3}\right], \\
\frac{\partial^2 q}{\partial x_l \partial x_m} &\equiv 0.
\end{align*}
Clearly $\frac{\partial^2 p}{\partial x_l \partial x_m} = \frac{\partial^2 p}{\partial x_l^2}$, for $l \le m$.
We define two matrices $\mathbf{A}$ and
$\mathbf{B}$ as following.
\[
\mathbf{A} = \begin{bmatrix}
1 & 1 & 1 \\
\frac{\partial p}{\partial x_1} & \frac{\partial p}{\partial x_2} & \frac{\partial p}{\partial x_3} \\
\frac{\partial q}{\partial x_1} & \frac{\partial q}{\partial x_2} & \frac{\partial q}{\partial x_3} \\
\end{bmatrix}, \quad
\mathbf{B} = \begin{bmatrix}
\frac{\partial^2 p}{\partial x_1^2} & \frac{\partial^2 p}{\partial x_1\partial x_2} & \frac{\partial^2 p}{\partial x_1\partial x_3} \\
\frac{\partial^2 p}{\partial x_1\partial x_2} & \frac{\partial^2 p}{\partial x_2^2} & \frac{\partial^2 p}{\partial x_2\partial x_3} \\
\frac{\partial^2 p}{\partial x_1\partial x_3} & \frac{\partial^2 p}{\partial x_2\partial x_3} & \frac{\partial^2 p}{\partial x_3^2} \\
\end{bmatrix}=\begin{bmatrix}
\frac{\partial^2 p}{\partial x_1^2} & \frac{\partial^2 p}{\partial x_1^2} & \frac{\partial^2 p}{\partial x_1^2} \\
\frac{\partial^2 p}{\partial x_1^2} & \frac{\partial^2 p}{\partial x_2^2} & \frac{\partial^2 p}{\partial x_2^2} \\
\frac{\partial^2 p}{\partial x_1^2} & \frac{\partial^2 p}{\partial x_2^2} & \frac{\partial^2 p}{\partial x_3^2} \\
\end{bmatrix}.
\]
It is easy to see that if $(\mathbf{x},\mathbf{y})$ is non-degenerate with $x_3>0$, then
$\frac{\partial^2 p}{\partial x_3^2}>\frac{\partial^2 p}{\partial x_2^2}>\frac{\partial^2 p}{\partial x_1^2}>0$. This implies that
$\mathbf{B}$ is positive definite.

For a vector $\mathbf{v}\in\mathbb{R}^3$ and $\epsilon>0$, we define $\mathbf{x'}$ by
\[
x_i'= \left\{\begin{array}{ll}
x_i + \epsilon v_i & \text{ if } i \le 3, \\
x_i & \text{ if } i > 3, \\
\end{array}\right.
\]

If $\mathbf{A}$ is invertible, let $\mathbf{v}$ be the (unique) vector for which
\[
\mathbf{A} \cdot \mathbf{v}^{T} = \begin{bmatrix}
0 \\
1 \\
1 \\
\end{bmatrix}.
\]
In particular $\sum_i x_i' = \sum_i x_i$. For $\epsilon$ sufficiently small,
\[
p(\mathbf{x}', \mathbf{y}) = p(\mathbf{x},\mathbf{y}) + \epsilon + O(\epsilon^2)>p(\mathbf{x}, \mathbf{y})\]\[
q(\mathbf{x}', \mathbf{y}) = q(\mathbf{x}, \mathbf{y}) + \epsilon> q(\mathbf{x}, \mathbf{y})
\]
contrary to the maximality of $(\mathbf{x},\mathbf{y})$.

If $\mathbf{A}$ is singular, pick some $\mathbf{v} \neq 0$ with $\mathbf{A} \cdot \mathbf{v}^{T} = \mathbf{0}$. Again $\sum_i x_i' = \sum_i x_i$. Since $\mathbf{B}$ is positive definite, for a sufficiently small $\epsilon$,
\[
p(\mathbf{x}', \mathbf{y}) = p(\mathbf{x},\mathbf{y})
+ \frac{\epsilon^2}{2}\cdot \mathbf{v} \cdot \mathbf{B} \cdot \mathbf{v}^{T} + O(\epsilon^3)>p(\mathbf{x},\mathbf{y}) \]\[
q(\mathbf{x}', \mathbf{y}) = q(\mathbf{x}, \mathbf{y}),
\]
Contradicting Lemma \ref{lemma_equality}.
\end{proof}

\begin{lemma}
\label{lemma_small1}
If $(\mathbf{x},\mathbf{y}) \in W_k$ is a non-degenerate maximum of $\varphi$ with $x_1> 0$, then $y_2=0$.
\end{lemma}

\begin{proof}
By Lemma \ref{lemma_small0} we may assume that $x_i=y_i=0$ for all $i \ge 3$. Suppose, towards contradiction, that $y_2\ne 0$. Let
\[
\mathbf{M} = \begin{bmatrix}
a_1 & a_2 \\
b_1 & b_2
\end{bmatrix},
\]
where
\begin{align*}
a_1 = \frac{\partial p}{\partial x_1} - \frac{\partial p}{\partial x_2} &= -s(s-1)\cdot y_1 \cdot x_2^{s-2},\quad &
b_1 = \frac{\partial q}{\partial x_1} - \frac{\partial q}{\partial x_2} &= r\cdot ((y_1 + y_2)^{r-1} - y_2^{r-1}), \\
a_2 = \frac{\partial p}{\partial y_1} - \frac{\partial p}{\partial y_2} &= s\cdot x_2^{s-1}, \quad &
b_2 = \frac{\partial q}{\partial y_1} - \frac{\partial q}{\partial y_2} &= -r(r-1) \cdot x_2\cdot y_2^{r-2}, \\
\end{align*}
If $\mbox{rank}(\mathbf{M})=2$, then there is a vector $\mathbf{v}=\left(v_1 \atop v_2\right)$ such that $\mathbf{M}\cdot \mathbf{v}= \left(1 \atop 1\right)$.
Define $x'_1=x_1+\epsilon v_1, x'_2=x_2-\epsilon v_1$ and  $y'_1=y_1+\epsilon v_2, y'_2=y_2-\epsilon v_2$.
Then $x'_1+x'_2+y'_1+y'_2=1$ and for sufficiently small $\epsilon>0$
\begin {eqnarray*}
p(\mathbf{x}', \mathbf{y}') &= & p(\mathbf{x},\mathbf{y})+\epsilon\Big(\frac{\partial p}{\partial x_1}v_1-\frac{\partial p}{\partial x_2}v_1+
\frac{\partial p}{\partial y_1}v_2-\frac{\partial p}{\partial y_2}v_2 \Big)+ O(\epsilon^2)\\
&= &p(\mathbf{x},\mathbf{y})+\epsilon\big(a_1v_1+a_2v_2\big)+ O(\epsilon^2)=p(\mathbf{x},\mathbf{y})+\epsilon+ O(\epsilon^2)>p(\mathbf{x},\mathbf{y}).
\end{eqnarray*}
Similarly $q(\mathbf{x}', \mathbf{y}')=q(\mathbf{x},\mathbf{y})+\epsilon+ O(\epsilon^2)>q(\mathbf{x},\mathbf{y})$.
Thus $(\mathbf{x},\mathbf{y})$ cannot be a maximum of $\varphi$.
Hence, $\mbox{rank}(\mathbf{M})\le 1$, and in particular
\[
\det \begin{bmatrix}
a_1 & b_1 \\
a_2 & b_2
\end{bmatrix} = 0,
\]
which implies that
\[
0 = x_2^{s-1}y_2^{r-1}\left((r-1)(s-1)\frac{y_1}{y_2} - \left(\frac{y_1}{y_2}+1\right)^{r-1} + 1\right), \\
\]
The function
$$g(\alpha) = (r-1)(s-1)\alpha-(\alpha + 1)^{r-1}+1$$
is strictly concave for $\alpha > 0$ and vanishes at $0$. Since $\alpha=0$ is not a maximum of $g$, the equation $g\left(\frac{y_1}{y_2}\right)=0$ determines $\frac{y_1}{y_2}$ uniquely.

Denote $\alpha=\frac{y_1}{y_2}$, and consider the following change of variables.
\begin{align*}
x_1' &= x_1 + \frac{1}{1 + (r-1)(s-1)\alpha}\cdot x_2,\quad &
x_2' &= \frac{(r-1)(s-1)\alpha}{1 + (r-1)(s-1)\alpha} \cdot x_2\\
y_1' &= y_1 + y_2 = (\alpha + 1) y_2,\quad &
y_2' &= 0
\end{align*}
Clearly, $x_1' + x_2' = x_1 + x_2$ and $y_1' + y_2' = y_1 + y_2$. Moreover,
\begin{align*}
q(\mathbf{x}', \mathbf{y}') &= (y_1')^{r} + r\cdot x_1' \cdot (y_1')^{r-1}\\
 &= (y_1 + y_2)^r + r\cdot x_1 \cdot (y_1 + y_2)^{r-1} + \frac{r\cdot x_2 \cdot (y_1 + y_2)^{r-1}}{1 + (r-1)(s-1)\alpha} \\
 &= (y_1 + y_2)^r + r\cdot x_1 \cdot (y_1 + y_2)^{r-1} + \frac{r\cdot (1+\alpha)^{r-1} \cdot x_2 \cdot y_2^{r-1}}{(1 + \alpha)^{r-1}} = q(\mathbf{x}, \mathbf{y}) \\
p(\mathbf{x}', \mathbf{y}') &= (x_1' + x_2')^{s} + s \cdot y_1' \cdot (x_2')^{s-1} \\
 &= (x_1 + x_2)^s + s\cdot (\alpha + 1)\cdot \left(\frac{(r-1)(s-1)\alpha}{1 + (r-1)(s-1)\alpha}\right)^{s-1}\cdot y_2 \cdot x_2^{s-1} \\
 &> (x_1 + x_2)^s + s\cdot \alpha \cdot y_2 \cdot x_2^{s-1} = p(\mathbf{x}, \mathbf{y}),
\end{align*}
Where the last inequality is a consequence of Lemma \ref{lemma_ineq} below. This contradicts Lemma \ref{lemma_equality}.
\end{proof}

\begin{lemma}
\label{lemma_ineq}
Let $r,s \ge 3$ be integers. Let $\alpha > 0$ be the unique positive root of
\[
(\alpha + 1)^{r-1} - 1 = (r-1)(s-1)\alpha.
\]
Then
\[
\left(1 + \frac{1}{(r-1)(s-1)\alpha}\right)^{s-1} < 1 + \frac{1}{\alpha}.
\]
\end{lemma}
\begin{proof}
First, we show that $(r-1) \alpha > 1$. Let $t = (r-1)\alpha$ and assume, by contradiction, that $t \le 1$. For $0 < t\le 1$, we have $e^t < 1 + 2t$. On the other hand,
$e \ge (1+\alpha)^{1/\alpha}$, implying $e^t \ge \left( 1 + \alpha\right)^{t / \alpha} = (1 + \alpha)^{r-1}$.
Thus we have $2t > (1+\alpha)^{r-1}-1 = (r-1)(s-1)\alpha = (s-1)t$, which implies $2 > s-1$, a contradiction. Therefore $(r-1)\alpha > 1$. Also, since
$1+x<e^x$ for all $x>0$, we have that $\big(1 + \frac{1}{(r-1)(s-1)\alpha}\big)^{s-1} < e^{\frac{1}{(r-1)\alpha}}$.
So it suffices to show that $e^{\frac{1}{(r-1)\alpha}} \le 1 + \frac{1}{\alpha}$. But since $(r-1)\alpha > 1$, we have
\[
\left(1 + \frac{1}{\alpha}\right)^{(r-1)\alpha} > 1 + \frac{(r-1)\alpha}{\alpha} = r \ge 3 > e,
\]
which finishes the proof of the lemma.
\end{proof}

\begin{lemma}
\label{lemma_small2}
If $(\mathbf{x},\mathbf{y}) \in W_k$ is a non-degenerate maximum of $\varphi$ with $x_1>0$, then $x_2=0$.
\end{lemma}
\begin{proof}
This proof is very similar to the proof of Lemma \ref{lemma_small1}. Now $x_1, x_2, y_1 > 0$ and $x_1 + x_2 + y_1 = 1$. Also
\begin{align*}
p(\mathbf{x},\mathbf{y}) &= (x_1 + x_2)^s + s \cdot y_1 \cdot x_2^{s-1}, \\
q(\mathbf{x},\mathbf{y}) &= y_1^r + r\cdot x_1 \cdot y_1^{r-1}.
\end{align*}
Let
\[
\mathbf{M} = \begin{bmatrix}
a_1 & a_2 \\
b_1 & b_2
\end{bmatrix},
\]
where
\begin{align*}
a_1 = \frac{\partial p}{\partial x_1} - \frac{\partial p}{\partial x_2} &= -s(s-1) \cdot y_1 \cdot x_2^{s-2}, \quad &
b_1 = \frac{\partial q}{\partial x_1} - \frac{\partial q}{\partial x_2} &= r\cdot y_1^{r-1}, \\
a_2 = \frac{\partial p}{\partial y_1} - \frac{\partial p}{\partial x_1} &= -s\cdot ((x_1+x_2)^{s-1} - x_2^{s-1}), \quad &
b_2 = \frac{\partial q}{\partial y_1} - \frac{\partial q}{\partial x_1} &= r(r-1)\cdot x_1\cdot y_1^{r-2}, \\
\end{align*}
If $\mathbf{M}$ is nonsingular, then there is a vector $\mathbf{v}=\left(v_1 \atop v_2\right)$ such that $\mathbf{M}\cdot \mathbf{v}= \left(1 \atop 1\right)$.
Define $x'_1=x_1+\epsilon (v_1-v_2), x'_2=x_2-\epsilon v_1$ and  $y'_1=y_1+\epsilon v_2$.
Then $x'_1+x'_2+y'_1=1$ and for sufficiently small $\epsilon>0$
\begin{eqnarray*}
p(\mathbf{x}', \mathbf{y}') &=&  p(\mathbf{x},\mathbf{y})+\epsilon\Big(\frac{\partial p}{\partial x_1}(v_1-v_2)-\frac{\partial p}{\partial x_2}v_1+
\frac{\partial p}{\partial y_1}v_2\Big)+ O(\epsilon^2)\\
&=& p(\mathbf{x},\mathbf{y})+\epsilon\big(a_1v_1+a_2v_2\big)+ O(\epsilon^2)=p(\mathbf{x},\mathbf{y})+\epsilon+
O(\epsilon^2)>p(\mathbf{x},\mathbf{y}).
\end{eqnarray*}
Similarly $q(\mathbf{x}', \mathbf{y}')=q(\mathbf{x},\mathbf{y})+\epsilon+ O(\epsilon^2)>q(\mathbf{x},\mathbf{y})$ and therefore
$(\mathbf{x},\mathbf{y})$ cannot be a maximum of $\varphi$. Hence,
\[
\det \begin{bmatrix}
a_1 & b_1 \\
a_2 & b_2
\end{bmatrix} =  0,
\]
which implies
\[
0 = y_1^{r-1}x_2^{s-1}\left((r-1)\cdot (s-1) \cdot\frac{x_1}{x_2} - \left(\frac{x_1}{x_2}+1\right)^{s-1} + 1\right).
\]
Let $\gamma = \frac{x_1}{x_2} > 0$. Then $1+(r-1)(s-1)\gamma-(1+\gamma)^{s-1}=0$ and concavity of the left hand side
shows that $\gamma$ is determined uniquely by this equation. Now make the following substitution:
\begin{align*}
x_1' &= 0\\
x_2' &= x_1 + x_2 = (1 + \gamma) \cdot x_2\\
y_1' &= \frac{1}{1 + (r-1)(s-1)\gamma}\cdot y_1\\
y_2' &= \frac{(r-1)(s-1)\gamma}{1 + (r-1)(s-1)\gamma} \cdot y_1
\end{align*}
Clearly $x_1' + x_2' = x_1 + x_2$ and $y_1' + y_2' = y_1$. Since $(1+\gamma)^{s-1}=1+(r-1)(s-1)\gamma$, we have
\begin{align*}
p(\mathbf{x}', \mathbf{y}') &= (x_2')^{s} + s\cdot y_1' \cdot (x_2')^{s-1}\\
 &= (x_1 + x_2)^s + s\cdot y_1\cdot x_2^{s-1} = p(\mathbf{x},\mathbf{y})\\
q(\mathbf{x}', \mathbf{y}') &= (y_1' + y_2')^{r} + r\cdot x_2' \cdot (y_2')^{r-1} \\
 &= y_1^r + r\cdot \frac{(1+\gamma)}{\gamma}\cdot\left(\frac{(r-1)(s-1)\gamma}{1 + (r-1)(s-1)\gamma}\right)^{r-1}\cdot x_1 \cdot y_1^{r-1} \\
 &> y_1^r + r\cdot x_1 \cdot y_1^{r-1} = q(\mathbf{x}, \mathbf{y}),
\end{align*}
Where the last inequality follows from Lemma \ref{lemma_ineq}, with $r$ and $s$ switched. Again, this contradicts Lemma \ref{lemma_equality}.
\end{proof}

By combining Lemmas \ref{lemma_equality} -- \ref{lemma_small2}, we obtain a proof of Lemma \ref{lemma_mainLemma}, which states that the maximum of $\varphi$ is attained by a
non-degenerate $(\mathbf{x},\mathbf{y})$ supported only on either $x_1,y_1$ or $y_1,x_2$. In the first case, let $x_1=\alpha$ and $y_1=1-\alpha$. Then
by Lemma \ref{lemma_equality}, $a\cdot p(\mathbf{x},\mathbf{y})=a\cdot \alpha^s=b\cdot q(\mathbf{x},\mathbf{y})=b\big[(1-\alpha)^r+r\alpha(1-\alpha)^{r-1}\big]$ and
$\varphi(\mathbf{x},\mathbf{y})=a \cdot \alpha^s$. In the second case, let $y_1=\beta$ and $x_2=1-\beta$. Then
$b\cdot q(\mathbf{x},\mathbf{y})=b\cdot \beta^r=a \cdot p(\mathbf{x},\mathbf{y})=a\big[(1-\beta)^s+s(1-\beta)^{s-1}\big]$ and $\varphi(\mathbf{x},\mathbf{y})=b \cdot \beta^r$. This
shows that the maximum of $\varphi$ is $\max \{a \cdot \alpha^{s}, b \cdot \beta^{r}\}$ with $\alpha, \beta$ satisfying the above equations. In terms of the original graph, this
proves that $\varphi$ is maximized by a graph of the form $Q_{n,t}$ or $\overline Q_{n,t}$, respectively. In particular, our problem has at most two extremal configurations (in some cases a clique and the complement of a clique can give the same value of $\varphi$).

\section{Stability analysis} \label{section_stability}
In this section we discuss the proof of Theorem \ref{stabilitytheorem}.
In essentially the same way that Theorem \ref{maintheorem_unrestricted} implies Theorem \ref{maintheorem}, this theorem follows from a stability version of Theorem \ref{maintheorem_unrestricted}:

\begin{theorem}
\label{stabilitytheorem_unrestricted}
Let $r, s \ge 3$ be integers and let $a,b > 0$ be real. For every $\epsilon > 0$, there exists $\delta > 0$ and an integer $N$ such that every $n$-vertex $G$ with $n > N$ for which
$$f(G) \ge \max \{a \cdot \alpha^{s}, b \cdot \beta^{r}\} - \delta$$
is $\epsilon$-close to some graph in $\mathcal{Q}_n$. Here $f, \alpha$ and $\beta$ are as in Theorem \ref{maintheorem_unrestricted}.
\end{theorem}
\begin{proof}
If $G$ is a threshold graph, the claim follows easily from Lemma \ref{lemma_mainLemma}. Since $G$ is a threshold graph, $f(G) = \varphi(\mathbf{x},\mathbf{y})+o(1)$ for some $(\mathbf{x},\mathbf{y})\in W_k$ and some integer $k$. As this lemma shows, the continuous function $\varphi$ attains its maximum on the compact set $W_k$ at most twice, and this in points that correspond to graphs from $\mathcal{Q}_n$. Since $f(G)$ is $\delta$-close to the maximum, it follows that $(\mathbf{x},\mathbf{y})$ must be $\epsilon'$-close to at least one of the two optimal points in $W_k$. This, in turn implies $\epsilon$-proximity of the corresponding graphs.

For the general case, we use the stability version of the Kruskal-Katona theorem due to Keevash \cite{keevash}. Suppose $G$ is a large graph such that $f(G) \ge \max \{a \cdot \alpha^{s},
b \cdot \beta^{r}\} - \delta$. Let $G_1$ be the shifted graph obtained from $G$. Thus $G_1$ is a threshold graph with the same edge density as $G$, and $f(G_1) \ge f(G)$ by Corollary \ref{corollary_increase}. Pick a small $\epsilon' > 0$. We just saw that for $\delta$ sufficiently small, $G_1$ is $\epsilon'$-close to $G_{max} \in \mathcal{Q}_n$. As we know, either $G_{max}= Q_{n,t}$ or $G_{max}= \bar Q_{n,t}$ for some $0 < t \le n$. We deal with the former case, and the second case can be done similarly. Now $|d(K_2;G) - d(K_2;G_{max})| \le \epsilon'$, since $G$ and $G_1$ have the same edge density. Moreover, $d(K_s; G) \ge d(K_s;G_{max}) -\delta/a$, because $f(G) \ge f(G_{max}) - \delta$.
Since $G_{max}$ is a clique, it satisfies the Kruskal-Katona inequality with equality. Consequently $G$ has nearly the maximum possible $K_s$-density for a given number of edges. By choosing
$\epsilon'$ and $\delta$ small enough and applying Keevash's stability version of Kruskal-Katona inequality, we conclude that $G$ and $G_{max}$ are $\epsilon$-close.
\end{proof}

\section{Second proof}
\label{section_shift}
In this section we briefly present the main ingredients for an alternative approach to Theorem \ref{maintheorem}. We restrict ourselves to the case $r=s$. This proof reduces the
problem to a question in the calculus of variations. Such calculations occur often in the context of shifted graphs.

Let $G$ be a shifted graph with vertex set $[n]$ with the standard order. Then, there is some $n \ge i \ge 1$ such that $A=\{1,...,i\}$ spans a clique, whereas $B=\{i+1,...,n\}$ spans
an independent set. In addition, there is some non-increasing function $F:A\rightarrow B$ such that for every $j \in A$ the highest index neighbor of $j$ in $B$ is $F(j)$, and all vertices of $B$ up to index $F(j)$ are connected to $j$.
Let $x$ be the relative size of $A$ and $1-x$ the relative size of $B$. In this case we can express (up to a negligible error term)
\begin{eqnarray*}
d(\overline{K}_k; G)&=&{n \choose k}^{-1}\left[{(1-x)n \choose k} + \sum_{1 \leq j \leq xn} {n-F(j) \choose k-1}
\right]=(1-x)^k+\frac{k}{n}\sum_{1 \leq j \leq xn}\left(\frac{n-F(j)}{n}\right)^{k-1}\\
&=&
(1-x)^k+kx(1-x)^{k-1}\sum_{1 \leq j \leq xn} \frac{1}{nx}\left(1-\frac{F(j)-xn}{(1-x)n}\right)^{k-1}.
\end{eqnarray*}
Let $f$ be a non-increasing function $f:\I\rightarrow\I$ such that $f(t)=\frac{F(j)-xn}{(1-x)n}$ for every $\frac{j-1}{xn} \leq t \leq \frac{j}{xn}$
(Think of $f$ as a relative version of $F$ both on its domain with respect to $A$ and its codomain with respect to $B$). Then
we can express $d(\overline{K}_k; G)$ in terms of $x$ and $f$
$$d(\overline{K}_k; G)=(1-x)^k+kx(1-x)^{k-1}\int_{0}^{1}(1-f(t))^{k-1}dt=\pa{G_{x,f}}.$$
Similarly one can show that
$$d(K_k; G)=x^k+kx^{k-1}(1-x)\int_{0}^{1}(k-1)t^{k-2}f(t)dt=\pc{G_{x,f}}.$$
Note that in this notation, $x=\theta$, $f=0$ (resp. $x=1-\theta$, $f=1$) corresponds to $Q_{n, \theta \cdot n}$, (resp. $\overline Q_{n, \theta \cdot n}$).

To prove Theorem \ref{maintheorem} for the case $r=s=k$, we show that assuming $\pc{G_{x,f}}\geq \alpha$, the maximum of $\pa{G_{x,f}}$ is attained for
either $f=0$ or $f=1$. For this purpose, we prove upper bounds on the integrals.
\begin{lemma}\label{lemma_integrals}
If $f:\I\rightarrow\I$ is a non-increasing function, then
\begin{equation*}
\int_{0}^{1}(1-f(t))^{k-1}dt\leq\max\left\{ 1-\left(\int_{0}^{1}(k-1)t^{k-2}f(t)dt\right)^{\frac{1}{k-1}},\left(1-\int_{0}^{1}(k-1)t^{k-2}f(t)dt\right)^{k-1}\right\}.
\end{equation*}
\end{lemma}
The bounds in Lemma \ref{lemma_integrals} are tight. Equality with the first term holds for $f$ that takes only the values $1$ and $0$, and equality with the second term occurs for $f$ a constant function. Proving Theorem \ref{maintheorem} for such functions is done using rather standard (if somehow tedious) calculations. Lemma \ref{lemma_integrals} itself is reduced to the following lemma through a simple affine transformation and normalization.

What non-decreasing function in $\I$ minimizes the inner product with a given monomial?
\begin{lemma}\label{lemma_reduced}
Let $g:\I\rightarrow[0,B]$ be a non-decreasing function with $B\geq 1$ and $\|g\|_{k-1}=1$. Then
\begin{equation*}
\langle (k-1)t^{k-2},g \rangle=\int_{0}^{1}(k-1)t^{k-2}g(t)dt\geq\min\left\{B\left(1-\left(1-\frac{1}{B^{k-1}}\right)^{k-1}\right),1\right\}.
\end{equation*}
Equality with the first term holds for
\[
g(t)=\left\{
\begin{matrix}
0 & t<1-\frac{1}{B^{k-1}}\\
B & t\geq 1-\frac{1}{B^{k-1}}
\end{matrix}
\right.
\]
The second equality holds for $g=1$.
\end{lemma}
We omit the proof which is based on standard calculations and convexity arguments.

\section{Shifting in hypergraphs}\label{section_hyper}

In this section, we will discuss a possible extension of Lemma \ref{lemma_increase} to hypergraphs. Consider two set systems $\mathcal{F}_1$ and $\mathcal{F}_2$ with vertex sets $V_1$
and $V_2$ respectively. A (not necessarily induced) \textit{labeled copy of $\mathcal{F}_1$ in $\mathcal{F}_2$} is an injection $I:V_1 \to V_2$ such that $I(F) \in \mathcal{F}_2$ for
every $F \in \mathcal{F}_1$. We denote by $\Cop(\mathcal{F}_1;\mathcal{F}_2)$ the set of all labeled copies of $\mathcal{F}_1$ in $\mathcal{F}_2$ and let
$$t(\mathcal{F}_1; \mathcal{F}_2):= |\Cop(\mathcal{F}_1;\mathcal{F}_2)|.$$
Recall that a vertex $u$ \textit{dominates} vertex $v$ if $S_{v\to u}(\mathcal{F})=\mathcal{F}$.
If either $u$ dominates $v$ or $v$ dominates $u$ in a family $\mathcal{F}$, we call the pair $\{u, v\}$ \textit{stable} in $\mathcal{F}$. If every pair is stable in
$\mathcal{F}$, then we call $\mathcal{F}$ a \textit{stable set system}.
\begin{theorem}
\label{theorem_increase}
Let $\mathcal{H}$ be a stable set system and let $\mathcal{F}$ be a set system. For every two vertices $u, v$ of $\mathcal{F}$ there holds
$$t(\mathcal{H}; S_{u\to v}(\mathcal{F})) \ge t(\mathcal{H}; \mathcal{F}).$$
\end{theorem}
\begin{corollary}
Let $G$ be an arbitrary graph and let $H$ be a threshold graph $H$. Then
$$t(H; S_{u\to v}(G)) \ge t(H; G),$$
for every two vertices $u, v$ of $G$.
\end{corollary}
\begin{proof}[Proof of Theorem \ref{theorem_increase} (sketch)]
We define a new shifting operator $\tilde{S}_{u\to v}$ for sets of labeled copies. First, for every $u,v\in V$, and a labeled copy $I:U\to V$, define $I_{u\leftrightarrow v}: U\to V$ by
\[
I_{u\leftrightarrow v}(w) = \left\{\begin{array}{ll}
I(w) & \text{ if } I(w) \ne u, v,\\
v & \text{ if } I(w) = u, \\
u & \text{ if } I(w) = v
\end{array}\right.
\]
For $\mathcal{I}$ a set of labeled copies, $I\in\mathcal{I}$, we let
\[
\tilde{S}_{u\to v}(I, \mathcal{I}) = \left\{\begin{array}{ll}
I_{u\leftrightarrow v} & \text{if } I_{u\leftrightarrow v}\not\in \mathcal{I} \text{ and } \textup{Im}(I) \cap \{u,v\} = \{u\},\\
I_{u\leftrightarrow v} & \text{if } I_{u\leftrightarrow v}\not\in \mathcal{I}, \{u,v\}\subset \textup{Im}(I), \text{ and } I^{-1}(u) \text{ dominates } I^{-1}(v) \text{ in } \mathcal{H},\\
I & \text{otherwise}.
\end{array}\right.
\]
Finally, let $\tilde{S}_{u\to v}(\mathcal{I}):= \{\tilde{S}_{u\to v}(I,\mathcal{I}) : I \in \mathcal{I}\}$.
Clearly, $|\tilde{S}_{u\to v}(\mathcal{I})| = |\mathcal{I}|$, and we prove that
$$\tilde{S}_{u\to v}(\Cop(\mathcal{H};\mathcal{F}))\subseteq \Cop(\mathcal{H};S_{u\to v}(\mathcal{F}))$$
thereby proving that $t(\mathcal{H};S_{u\to v}(\mathcal{F})) \ge t(\mathcal{H}; \mathcal{F})$. As often in shifting, the proof is done by careful case analysis which is omitted.
\end{proof}

\section{Concluding remarks}\label{section_concluding}
In this paper, we studied the relation between the densities of cliques and independent sets in a graph. We showed that if the
density of independent sets of size $r$ is fixed, the maximum density of $s$-cliques is achieved when the graph itself is either a clique on a subset of the vertices, or a
complement of a clique.
On the other hand, the problem of minimizing the clique density seems much harder and has quite different extremal graphs for various values of $r$ and $s$ (at least when
$\alpha=0$, see \cite{das-et-al, pikhurko}).

\begin{question}
Given that $d(\overline{K}_r; G) = \alpha$ for
some integer $r \ge 2$ and real $\alpha \in [0,1]$, which graphs minimize $d(K_s; G)$?
\end{question}
In particular, when $\alpha=0$ we ask for the least possible density of $s$-cliques in graphs with independence number $r-1$.
This is a fifty-year-old  question of Erd\H{o}s, which is still widely open. Das et al \cite{das-et-al}, and independently Pikhurko
\cite{pikhurko}, solved this problem for certain values of $r$ and $s$. It would be interesting if one could describe how the extremal graph changes as $\alpha$ goes from $0$ to $1$ in
these cases. As mentioned in the introduction, the problem of minimizing $d(K_s; G)$ in graphs with fixed density of $r$-cliques for $r<s$
is also open and so far solved only when $r=2$.\\

\vspace{0.2cm}
\noindent
{\bf Note added in proof.}
After writing this paper, we learned that P. Frankl, M. Kato, G. Katona and N. Tokushige \cite{frankl-kato-katona-tokushige} independently considered the same problem and obtained
similar results when $r=s$.

\vspace{0.2cm}
\noindent
{\bf Acknowledgment.}
We would like to thank the anonymous referee for valuable comments and suggestions which improve the presentation of the paper.

\end{document}